\documentclass{article}
\usepackage[pdftex]{hyperref}
\hypersetup{
  colorlinks   = true, 
  urlcolor     = gray, 
  linkcolor    = red, 
  citecolor   =  blue
}
\usepackage{amsmath,amsfonts,amsthm,amssymb,graphicx,xcolor}
\usepackage[all,2cell,ps]{xy}

\theoremstyle{plain}
\newtheorem{theorem}[equation]{Theorem}
\newtheorem{corollary}[equation]{Corollary}
\newtheorem{lemma}[equation]{Lemma}
\newtheorem{proposition}[equation]{Proposition}
\newtheorem{conjecture}[equation]{Conjecture}

\theoremstyle{definition}
\newtheorem{define}[equation]{Definition}

\newtheorem{remark}[equation]{Remark}

\newtheorem{question}[equation]{Question}

\newcommand{\p}{\partial}

\newcommand{\IH}{\mathbb{H}}

\newcommand{\IQ}{\mathbb{Q}}
\newcommand{\IR}{\mathbb{R}}

\newcommand{\tr}{\mathrm{tr}}

\newcommand{\ssm}{\smallsetminus}

\title{Price Inequalities and Betti Number Growth on Manifolds without Conjugate Points}
\author{\small{Luca F. Di Cerbo}\footnote{Partially supported by a grant associated to the S. S. Chern position at ICTP.} \\ \scriptsize{University of Florida}\\ \footnotesize{\textsf{ldicerbo@ufl.edu}} \and \small{Mark Stern}\footnote{Partially supported by Simons Foundation Grant 3553857} \\ \scriptsize{Duke University} \\ \footnotesize{\textsf{stern@math.duke.edu}}}
\date{}

\begin{document}

\maketitle

\begin{abstract}
We derive Price inequalities for harmonic forms on manifolds without conjugate points and with a negative Ricci upper bound. The techniques employed in the proof work particularly well for manifolds of non-positive sectional curvature, and in this case we prove a strengthened Price inequality. We employ these  inequalities to study the asymptotic behavior of the Betti numbers of coverings of Riemannian manifolds without conjugate points. Finally, we give a vanishing result for $L^{2}$-Betti numbers of closed manifolds without conjugate points.  
\end{abstract}

\vspace{8cm}

\tableofcontents

\vspace{1cm}

\section{Introduction}

In the early 1980's, Donnelly and Xavier introduced an integral inequality \cite[Theorem 2.2]{Donnelly}  which became a standard tool in the proof of cohomology vanishing theorems and spectral bounds on simply connected negatively curved manifolds. Around the same time  Price introduced an inequality \cite{Price} which became  ubiquitous in the study of singularities of harmonic maps and Yang-Mills fields. Price's inequality can be proved in a manner which generalizes the proof of Donnelly and Xavier's inequality. 
In this work, we further this circle of ideas by applying Price's generalization of the Donnelly-Xavier inequality to study cohomology, obtaining new bounds on the Betti numbers of Riemannian manifolds with negative Ricci curvature and no conjugate points. In particular, we obtain asymptotic bounds on the growth of Betti numbers in towers of regular coverings of such manifolds. The range of Betti numbers to which these results apply depends on how negative the Ricci curvature is assumed to be. 

Our approach is quite robust and works without assuming pinched negative curvature. In particular, the hypotheses of the main technical estimate, Theorem \ref{Price1}, allow some positive sectional curvature. Of course, when we impose stronger pinched negative sectional curvature assumptions, we derive a strengthened Price inequality, Theorem \ref{Price2}.

Lohkamp (\cite[Theorem A]{Lohkamp}) proved that every manifold of dimension $\geq 3$ admits a metric of negative Ricci curvature. Hence there are no topological consequences of the negative Ricci curvature assumption. Lohkamp's theorem does not, however, provide a mechanism for producing the large geodesic balls required for our estimates to be useful. With the no conjugate point hypothesis, we can find large geodesic balls by passing to finite covers with large injectivity radius.

The study of the asymptotic behavior of Betti numbers has attracted considerable interest in the last four decades, especially in the context of coverings of locally symmetric spaces of non-compact type.  See, for example,  \cite{Wallach}, \cite{Wallach1}, \cite{Xue}, \cite{Sarnak}, \cite{Bergeron}, \cite{Marshall}. Many of the algebraic techniques employed in these works are not obviously amenable to generalization outside the locally symmetric context.

Given a co-compact torsion free lattice $\Gamma$ acting on a symmetric space $G/K$ of non-compact type, we say a sequence of nested, normal, finite index subgroups $\{\Gamma_{i}\}$ of $\Gamma$ is  a \emph{cofinal filtration} of $\Gamma$ if  $\bigcap_{i}\Gamma_{i}$ is the identity element.   Any such $\Gamma$ is known to be \emph{residually finite} (see \cite[Proposition 2.3]{Borel}), so that cofinal filtrations always exists.  Denote by $M_{i}$ the finite index regular cover of $\Gamma\backslash G/K$ associated to $\Gamma_{i}$. It follows from the results of  \cite{Wallach}, \cite{Wallach1}, that for any cofinal filtration of $\Gamma$, for $k\not =\frac{1}{2}dim(G/K)$, 
\begin{align}\label{ratio}
\lim_{i\rightarrow\infty}\frac{b_{k}(M_{i})}{Vol(M_{i})}=0,
\end{align}
where  $b_{k}(M_{i})$ denotes the $k$-th Betti number of $M_{i}$. 

Our techniques provide new methods for proving such results and quantifying the sub volume order growth of certain Betti numbers. Because our techniques do not rely on representation theory or the trace formula, it is natural to consider these  questions on manifolds which are \emph{not} locally symmetric. The Price inequalities given in Theorem \ref{Price1} and Proposition \ref{-1ball} are tailored to address such problems. For example, we have the following \emph{non-locally symmetric} analog of the DeGeorge-Wallach  result \eqref{ratio}. 

\begin{theorem}[See Corollary \ref{intro2}]\label{DGW}
Let $(M^{n}, g)$ be a closed Riemannian manifold without conjugate points with $-1\leq \sec_{g}\leq 1$. 
Assume there exists $\delta>4k^2$  such that
\[
-Ric \geq \delta g.
\]
Let $\pi_i:M_i\to M$ be a sequence of Riemannian covers of $M$ with injectivity radii, denoted $\gamma_{g_{i}}(M_i)$, satisfying $\gamma_{g_{i}}(M_i)\to \infty$. 
Then there exists $b(n,k,\delta)>0$ so that for $\gamma_{g_{i}}(M_i)$ sufficiently large,    
\begin{align}\label{bdble29} \frac{b_k(M_i)}{Vol(M_i)}\leq   b(n, k, \delta)e^{-(\sqrt{\delta}-2k)\gamma_{g_{i}}(M_i)}.
\end{align}
In particular, 
\begin{align}\label{limble291} 
\lim_{i\to\infty}\frac{b_k(M_i)}{Vol(M_i)}=0.
\end{align}
\end{theorem}

Observe that the range of Betti numbers covered by Theorem \ref{DGW} grows as the square root of the lower bound $\delta$. Under our curvature normalization, $\delta$ can at most be $n-1$ with equality if and only if the underlying Riemannian manifold is real hyperbolic with sectional curvature $-1$. Thus, Theorem \ref{DGW} does not address the full range of Betti numbers achieved by DeGeorge and Wallach in the locally symmetric case. On the other hand, Theorem \ref{DGW} is not only free of any homogeneity requirement on the metric, but also does not require any direct assumption on the sectional curvature.  \\

Given a closed manifold $(M, g)$ without conjugate points and infinite residually finite fundamental group, the injectivity radius goes to infinity in any tower of regular Riemannian covers associated to a cofinal filtration of its fundamental group. (See  \cite[Theorem 2.1]{Wallach}.) Hence Theorem \ref{DGW} not only implies \eqref{ratio} for certain $k$-th Betti numbers, but also implies $\frac{b_{k}(M_i)}{Vol_{g}(M_i)}$  decays exponentially in the injectivity radii $\gamma_{g_{i}}(M_{i})$. Such decay results have been obtained by numerous authors in the locally symmetric context. See for example \cite{Xue}, \cite{Sarnak}, \cite{Marshall}. Outside the locally symmetric space context, related results have also been obtained by \cite{BLLS} for $p-$adically  defined towers, for which they show that if $\lim_{i\to\infty}\frac{b_{k}(M_i)}{Vol_{g}(M_i)} = 0,$ then $\frac{b_{k}(M_i)}{Vol_{g}(M_i)}$  decays like $Vol(M_i)^{-d}$, for a specific dimensional constant $d$.  

There is overlap between this paper and the results of  \cite{Clair} on  Betti number bounds.  The  results in \cite{Clair} are derived under some combination of a vanishing assumption on $L^{2}$-Betti numbers, a spectral gap assumption, and a positivity assumption for the Novikov-Shubin invariants. Our techniques not only do not require these latter assumptions but can be used  both to verify the requisite $L^{2}$-Betti number vanishing or spectral gap hypotheses of \cite{Clair}, and to give a new analytical proof of the consequences of these assumptions.   Price inequalities can also be applied to obtain information about Novikov-Shubin invariants. We hope to explore this direction in future work.

We can reexpress the decay results of Theorem \ref{DGW} in terms of volumes, under suitable hypotheses  on the injectivity radii of Riemannian coverings associated to a cofinal filtration. We introduce one such an assumption on the cofinal filtration which we call  ``congruence type'', see Definition \ref{congruence type}. Under this assumption, Theorem \ref{DGW} immediately yields the following corollary. 

\begin{corollary}\label{Answer1}
Let $(M^{n}, g)$ be a Riemannian manifold without conjugate points with $-1\leq\sec_{g}\leq 1$ and residually finite fundamental group $\Gamma$. Assume
\[
-Ric \geq \delta g, \quad \delta> 4k^2.
\] 
For any congruence type cofinal filtration $\{\Gamma_{i}\}$ of $\Gamma$ of exponent $\alpha$, if we denote by $\pi_{i}: M_{i}\rightarrow M$ the regular Riemannian cover of $M$ associated to $\Gamma_{i}$, we have for $\gamma_g(M)$ sufficiently  large,
\[
b_{k}(M_{i})\leq d(n,k, \delta, \Gamma) Vol(M_{i})^{1-2 \alpha(\frac{\sqrt{\delta}}{2}-k)}.
\]
where $d(n,k, \Gamma)$ is a positive constant. 
\end{corollary}

If one assumes the sectional curvature to be negative and suitably pinched, Corollary \ref{Answer1} can be considerably strengthened. See, for example, Corollary \ref{pinchedalpha}. We obtain our sharpest results for real hyperbolic space. See Theorem \ref{real hyperbolic}. The pinched cases, of course,  have already been well studied. See for example \cite{Donnelly} and \cite{Clair}. We provide a self-contained treatment of this important class of examples. In a subsequent paper, we will show how to modify our techniques to obtain stronger estimates for Betti numbers of the complex and quaternionic hyperbolic spaces than follow from our results in this paper.  \\ 

Finally, our techniques  can also be applied to $L^{2}$-cohomology problems. For example, we have the following vanishing result for $L^{2}$-Betti numbers of closed manifolds without conjugate points.

\begin{theorem}[See Theorem \ref{Betti-Vanishing}]\label{singer1}
Let $(M^{n}, g)$ be a closed Riemannian manifold without conjugate points and $-1\leq \sec_{g}\leq 1$. If there exists $\delta>4k^2$ such that
\[
-Ric\geq \delta g,
\]
then the $k$-th $L^{2}$-Betti number of $M$ vanishes. 
\end{theorem}

Theorem \ref{singer1} provides further evidence towards the Singer Conjecture for aspherical manifolds. For the details of the proof and more on the Singer Conjecture we refer to Section \ref{Atiyah}.\\ 

We conclude this introduction by noting that Theorem \ref{singer1}  applies to large classes of manifolds which are not locally symmetric. These include, for example, the exotic locally symmetric spaces constructed by Farrell-Jones \cite{Farrell}, the Gromov-Thurston manifolds \cite{Gromov}, the Mostow-Siu surfaces \cite{Mostow}, and the manifolds obtained by Anderson's higher dimensional Dehn filling \cite{Anderson}.\\ \\

\noindent\textbf{Acknowledgments}.   The authors thank the  International Centre for Theoretical Physics (ICTP) for the excellent working environment during the early stages of this collaboration. They also thank the referees for constructive comments on this work. The first author thanks the University of Florida for support during the final stages of this work. The second author thanks Michael Lipnowski for many helpful conversations about bounding Betti numbers of hyperbolic manifolds. \\


\section{Some Integral Equalities for Harmonic Forms}\label{preliminaries}

In this section, we fix notation and derive the Price/Donnelly-Xavier inequality for harmonic $k$-forms.

Let $(M^{n}, g)$ be a complete $n$-dimensional Riemannian manifold, with injectivity radius $\gamma_{g}(M)$. Given a point $p\in M$, denote by $B_{R}(p)$ the geodesic ball of radius $0<R\leq\gamma_{g}(M)$ centered at $p$. In such a ball, let $g=dr^{2}+g_{r}$ be the expression of the metric $g$ in geodesic polar coordinates, and let  $\partial_{r}$ denote the  unit radial vector field.

Given a $1$-form $\phi$, let $e(\phi)$ denote exterior multiplication on the left by $\phi$. Let $e^*(\phi)$ denote the adjoint operator. Fix  a local orthonormal frame $\{e_j\}_j$ and coframe $\{\omega^j\}_j$. Acting on forms of arbitrary degree,  the Lie derivative in the radial direction can be written as 
\begin{align}\label{lie}L_{\p_r}=\{d,e^*(dr)\} = \nabla_{\p_r}+e(\omega^j)e^*(\nabla_{e_j}dr).
\end{align} 
Choosing the orthonormal frame so that $e_n=\p_r$, we may express $\nabla_{e_j}dr$ in terms of the 
 the second fundamental form $h$ of $S_{r}(p)$. Recall for $X$ and $Y$ tangent to $S_r$, 
\[
h(X, Y):=g(\nabla_{X}\partial_{r}, Y).
\]
Thus,  we write
\[
\nabla_{e_{i}}dr=h_{i1}\omega^1+\ldots +h_{in-1}\omega^{n-1},\]
and 
\begin{align}\label{iidr}e(\omega^j)e^*(\nabla_{e_j}dr)= \sum_{j,k<n}h_{jk}e(\omega^j)e^*(\omega^k)=:Q.
\end{align}
The operator $Q$ defined by   \eqref{iidr} is the natural extension of the second fundamental form to an endomorphism of  forms of arbitrary degree.  As usual, we set the mean curvature of the geodesic sphere $S_{r}$ to be the trace $H=\sum_kh_{kk}$. 

Let $\mathcal{H}^{k}_{g}(M)$ denote the strongly harmonic $k$-forms on $M$. Given  $\alpha\in \mathcal{H}^{k}_{g}(M)$,  we have
\begin{align}\notag
\int_{B_{R}(p)}\langle L_{\partial_{r}}\alpha, \alpha\rangle dv&=\int_{B_{R}(p)}\langle d(e^*(dr)\alpha), \alpha\rangle dv \\ \notag
&=\int_{B_{R}(p)}\langle e^*(dr)\alpha, d^{*}\alpha\rangle dv+\int_{S_{R}(p)}|e^*(dr)\alpha|^{2}d\sigma \\  
&=\int_{S_{R}(p)}|e^*(dr)\alpha|^{2}d\sigma,\label{lie2}
\end{align}
where $d\sigma$ is the volume element on the geodesic sphere $S_{R}(p)$. Using \eqref{lie} for $L_{\p_r}$ gives the alternate expression
\begin{align}\notag
\int_{B_{R}(p)}\langle L_{\partial_{r}}\alpha, \alpha\rangle dv&=\int_{B_{R}(p)}\langle (\nabla_{\p_r}+Q)\alpha, \alpha\rangle dv   \\ \notag
&=\int_{B_{R}(p)} \frac{1}{2}(L_{\p_r}-Q)(|\alpha|^2dv)+\int_{B_{R}}\langle Q\alpha, \alpha\rangle dv  \\  
&=\int_{B_{R}(p)}\langle \big(Q-\frac{H}{2}\big)\alpha, \alpha\rangle dv + \int_{S_{R}(p)} \frac{1}{2}  |\alpha|^2d\sigma. \label{lie3}
\end{align}
Next, let us define the functions
\begin{align}\label{defmu}
\mu(r):=\frac{\int_{S_{r}(p)}|e^*(dr)\alpha|^{2}d\sigma}{\int_{S_{r}(p)}|\alpha|^{2}d\sigma},
\end{align} 
and
\begin{align}\label{defq}
q(r):=\frac{\int_{S_{r}(p)}\langle ( \frac{H}{2}-Q)\alpha, \alpha\rangle d\sigma}{\int_{S_{r}(p)}|\alpha|^{2}d\sigma}.
\end{align} 
Equating the expressions in \eqref{lie2} and \eqref{lie3} yields 
\begin{align} \int_{B_{R}(p)}q(r)|\alpha|^2 dv = \frac{1}{2} \int_{S_{R}(p)} ( 1-2\mu(R))|\alpha|^2d\sigma. \label{lie4}
\end{align}
Now if we multiply \eqref{lie4} by $\phi'(R)$, $\phi$ to be determined, and integrate from $\sigma$ to $\tau\leq\gamma_g(M)$ we get 
\begin{align}  \phi(\tau) \int_{B_{\tau}(p)}q(r)|\alpha|^2 dv
-\phi(\sigma) \int_{B_{\sigma}(p)}q(r)|\alpha|^2 dv\nonumber\\
= \int_{B_{\tau}(p)\setminus B_\sigma(p)}[\phi(r)q(r)   + \frac{1}{2}  \phi'(r)( 1-2\mu(r))]|\alpha|^2dv. \label{tbd}
\end{align}
Next, we choose 
\begin{align}\label{phidef}
\phi(r):= e^{-\int_\sigma^r\frac{q(s)ds}{\frac{1}{2}-\mu(s)}}.
\end{align}
in order to eliminate the last line of \eqref{tbd}. Let us summarize this discussion into a proposition.

\begin{proposition}\label{Price Equality}
Let $(M, g)$ be a Riemannian manifold. For any strongly harmonic $k$-form $\alpha\in\mathcal{H}^{k}_{g}(M)$ and for any $\sigma<\tau\leq\gamma_{g}(M)$, we have the Price equality
\begin{align} 
\phi(\sigma)\int_{B_{\sigma}(p)}q(r)|\alpha|^2 dv= \phi(\tau) \int_{B_{\tau}(p)}q(r)|\alpha|^2 dv, \label{price}
\end{align}
where $\mu(r)$, $q(r),$ and $\phi(r)$ are respectively defined as in \eqref{defmu}, \eqref{defq} and \eqref{phidef}.
\end{proposition}

We conclude this section by studying the behavior of the function $\mu(r)$ for $r$ close to zero.

\begin{lemma}\label{muzero}
Let $(M, g)$ be a Riemannian manifold and $p\in M$. Given a $k$-form $\alpha$ such that $\alpha(p)\not = 0$, we have
\begin{align}\lim_{r\to 0}\mu(r)=\frac{k}{n}.\end{align}
\end{lemma}

\begin{proof}This lemma does not require $\alpha$ to be harmonic. Fix geodesic coordinates with $p$ at the origin. Write 
\begin{align}\notag
|e^*(dr)\alpha|^2(x) &= \sum_{i,j}\frac{x^ix^j}{r^2}\langle e^*(dx^i)\alpha(x),e^*(dx^j)\alpha(x)\rangle \\ \notag
&= \sum_{i,j}\frac{x^ix^j}{r^2}\langle e^*(dx^i)\alpha(0),e^*(dx^j)\alpha(0)\rangle+o(1).\end{align}
Using 
\begin{align}\notag
\int_{S_r}\frac{x^ix^j}{r^2}d\sigma = \frac{\delta_{ij}}{n}Vol(S_r),
\end{align}
and
\[
\sum_j|e^*(dx^j)\alpha|^2(0) = k|\alpha|^2(0),
\]
we see 
\[
\mu(r) = \frac{k}{n}+o(1).
\]
\end{proof}

In order to extract geometric information out of Proposition \ref{Price Equality}, we need to understand the positivity properties of $q(r)$. This is a problem in comparison geometry which we address in the next section. 

\section{Controlling the Second Fundamental Form} \label{Ricci Bound} 

Let $ \lambda_1\geq \cdots \geq \lambda_{n-1}$ denote the eigenvalues of $h_{ij}$. $Q$ commutes with the decomposition of a $k-$form $\alpha$ as $\alpha = e^*(dr)e(dr)\alpha + e(dr)e^*(dr)\alpha .$  Hence
\begin{align}\label{ptwise}\langle ( \frac{H}{2}-Q)\alpha, \alpha\rangle&\geq \frac{1}{2}(-\lambda_1-\cdots-\lambda_k+\lambda_{k+1}+\cdots+\lambda_{n-1})|e(dr)\alpha |^2\notag\\
& + \frac{1}{2}(-\lambda_1-\cdots-\lambda_{k-1}+\lambda_{k}+\cdots+\lambda_{n-1})|e^*(dr)\alpha|^2.
\end{align}

With this notation,
\begin{align} \label{fundamental}  \int_{S_{R}(p)} q(r) | \alpha|^2d\sigma \geq  \int_{S_{R}(p)}((\frac{H}{2}-  \sum_{i=1}^{k}\lambda_i)|\alpha|^2  +\mu\lambda_k| \alpha|^2) d\sigma.
\end{align}
In order to extract information from this inequality, we use the Rauch comparison theorem (see eg. \cite[p. 255]{Petersen})  and the  Riccati Equation for the mean curvature of a geodesic sphere (see eg. \cite[Chapter 5]{Petersen})  to control the second fundamental form terms arising in the main inequality of Proposition \ref{fundamental}. We first recall  those results.  

Let $\gamma(K)$ denote the injectivity radius of the space form of constant curvature $K$. 
Let
\begin{equation*}
s_{K}(r) := \left\{ \begin{array}{rl} 
\frac{1}{\sqrt{K}}\sin(\sqrt{K}r) & \text{if } K> 0,\\ 
r & \text{if } K = 0,\\ 
\frac{1}{\sqrt{-K}}\sinh(\sqrt{-K}r)& \text{if } K< 0. 
\end{array} \right. 
\end{equation*}
\begin{theorem}[Rauch Comparison]\label{comparison}
If $(M^{n}, g)$ satisfies $K_{1}\leq \sec_{g}\leq K_{2}$, $h(r)$ is the second fundamental form of the geodesic sphere of radius $r$, and $g=dr^{2}+g_{r}$ denotes the metric in geodesic polar coordinates,  then 
\[
\frac{s'_{K_{2}}(r_1)}{s_{K_{2}}(r_1)}g_{r_1}\leq h(r_1), \text{ and  }h(r_2)\leq \frac{s'_{K_{1}}(r_2)}{s_{K_{1}}(r_2)}g_{r_2},
\]
for $0<r_1\leq \min\{\gamma_g(M), \gamma(K_2)\},$ and $0<r_2\leq \min\{\gamma_g(M), \gamma(K_1)\}.$
\end{theorem}

If we assume  
\begin{align}\label{normalize}
-1\leq\sec_{g}\leq \kappa,
\end{align}   Rauch's comparison   becomes:
\begin{align}\label{Rauch1}
 \sqrt{\kappa}\cot( \sqrt{\kappa} r_1)g_{r_1}\leq h(r_1),\text{ and  } h(r_2)\leq \coth(r_2)g_{r_2},
\end{align}
for $0<r_1\leq \min\{\gamma_g(M), \frac{\pi}{2 \sqrt{\kappa}}\},$ and $0<r_2\leq  \gamma_g(M) ,$
which then implies the following bounds on the mean curvature $H(r,\sigma)$ of any geodesic sphere $S_{r}.$
\begin{align}\label{Rauch2}
(n-1) \sqrt{\kappa}\cot( \sqrt{\kappa} r_1)\leq H(r_1, \sigma),\text{ and  } H(r_2, \sigma)\leq (n-1)\coth(r_2).
\end{align}

The Riccati equation for the mean curvature $H$ is:
\begin{align}
\partial_{r}H+|h|^2=-Ric(\partial_{r}, \partial_{r}).
\end{align}
Since $H^2\geq |h|^2$, we have
\begin{align}\label{11}
\partial_{r}H+H^2\geq-Ric(\partial_{r}, \partial_{r}).
\end{align}
If we assume the Ricci curvature is negative and bounded away from zero, say
\begin{align}\label{ricciass}
-Ric \geq \delta g, \quad (n-1)\geq\delta>0,
\end{align}
then \eqref{11} implies
\begin{align}\label{Riccati}
\partial_{r}H+H^{2}\geq \delta.
\end{align}(The upper bound $(n-1)$ on $\delta$ follows from the normalization \eqref{normalize} of the sectional curvature.) 
Next consider the ordinary differential equation  saturating the inequality given in equation \eqref{Riccati}
\[
u'+u^{2}=\delta.
\]
One solution to this equation is given by
\begin{align}\label{u}
 u(r):=\sqrt{\delta}\coth(\sqrt{\delta}\,r).
\end{align}
We now use a Riccati type comparison argument in conjunction with Rauch's comparison.  
\begin{lemma}\label{Riccati-Rauch}
Let $(M^{n}, g)$, $n\geq 3$, be a closed Riemannian manifold with $-1\leq\sec_{g}\leq \kappa$, and 
\[
-Ric\geq \delta g, \quad (n-1)\geq\delta>0. 
\]
 For any $p\in M$ and geodesic sphere $S_{r}(p)$ with $r\leq \gamma_{g}(M)$, we have
\[
H(r, \sigma)\geq\sqrt{\delta}\coth(\sqrt{\delta}\,r),\,\,\forall \sigma\in S_{r}(p).
\]
\end{lemma} 

\begin{proof}
 On any common interval of definition for $H$ and $u$, equations \eqref{Riccati} and \eqref{u} imply  the following first inequality:
\[
\big((u-H)e^{\int(u+H)}\big)'\leq 0.
\]
Thus, if we can find $r_{0}$ such that $H(r_{0}, \sigma)\geq u (r_{0})$, we then have that $H(r, \sigma)\geq u(r)$
for any $r\geq r_{0}$. Now, recall from equation \eqref{Rauch2} that
\begin{align}\label{hbnd}
(n-1)\sqrt{\kappa}\cot(\sqrt{\kappa} r)\leq H(r, \sigma)\leq (n-1)\coth(r),
\end{align}
for $r\leq\max(\gamma_{g}(M), \frac{\pi}{2\sqrt{\kappa}})$. Observe  $(n-1)\sqrt{\kappa}\cot(\sqrt{\kappa} r) = \frac{n-1}{r} +O(r)$ and  $\sqrt{\delta }\coth(\sqrt{\delta }r)= \frac{1}{r}+O(r)$. Thus, for $r'$ sufficiently close  to zero, we have $H(r', \sigma)\geq \sqrt{\delta}\coth(\sqrt{\delta}r')$,   $\forall \sigma\in S_{r'}$,  and the result follows.
\end{proof}

\section{An Integral Inequality for Harmonic Forms}

The main result of this section is the inequality given in Theorem \ref{Price1}. This inequality controls the the $L^{2}$-norm of a harmonic $k$-form on a ball of fixed radius, in terms of the $L^{2}$-norm on the complement of the given ball.  
We begin with a lemma giving pointwise and integral bounds on the geometric quantity $q(r)$ appearing in the Price equality given in Proposition \ref{Price Equality}. Key to this lemma are the Riccati-Rauch arguments of Section \ref{Ricci Bound}.

\begin{lemma}\label{Riccati-Rauch1}
Let $(M^{n}, g)$ be a Riemannian  manifold with $-1\leq\sec_{g}\leq \kappa$ and $0\leq\kappa\leq 1$. Let $\alpha$ be a harmonic $k-$form on $M$. Assume further that
\[
-Ric \geq \delta g, \quad (n-1)\geq\delta>4k^2. 
\]
Then for $0\leq \sigma<r<\gamma_{g}(M),$ 
\begin{align}\label{intbound}\int_\sigma^rq(s)ds\geq (r-\sigma)\Big(\frac{\sqrt{\delta}}{2}-k\Big)+k\ln(1-e^{-2\sigma}).
\end{align}
 Define $r_0:=\frac{1}{\sqrt{\kappa}}arccot(\frac{\sqrt{\delta}}{\sqrt{\kappa}(n-1)})\in(0, \frac{\pi}{2\sqrt{\kappa}})$. If
\begin{align}\label{r0}
\frac{\sqrt{\delta}}{2k}\geq \coth(r_0)+\frac{\epsilon}{k},
\end{align}
then 
\begin{align}\label{qbnd}\epsilon\leq q(r)\leq (n-1)\coth(r),
\end{align}  
for any $p\in M$ and $r\leq\gamma_{g}(M)$.
\end{lemma} 

\begin{proof}
Recall equation \eqref{Rauch1} gives the upper bound on the second fundamental form $h$ of the geodesic sphere $S_{r}$: 
\[
h\leq\coth(r)g_{r},
\]
for all $r\leq\gamma_{g}(M)$.  Then Proposition \ref{fundamental} yields
\begin{align}\notag
q(r) \geq  \frac{H}{2}-k\coth(r)+\mu(r)\coth(r) \geq  \frac{H }{2}-k\coth(r).
\end{align}
The lower bound on $h$ in \eqref{Rauch1} and the monotonicity of $\frac{\cot(\sqrt{\kappa}\, r)}{\coth(r)}$ on $(0,\frac{\pi}{2\sqrt{\kappa}})$, implies that given $r_{0}<\frac{\pi}{2\sqrt{\kappa}}$ for any $r\in (0,r_0]$ we have
\begin{align}\label{smallr}
 \frac{H(r, \sigma)}{2}-k\coth(r) &\geq    \frac{(n-1)\sqrt{\kappa} \cot(\sqrt{\kappa} r)}{2 } -k\coth(r)\nonumber\\
& \geq  k\coth(r_0)\Big(\frac{(n-1)}{2k}\frac{\sqrt{\kappa}\cot(\sqrt{\kappa}  r_0)}{\coth(r_0)}-1\Big)\geq \epsilon.
\end{align}
 On the other hand by Lemma \ref{Riccati-Rauch} and the monotonicity of $\coth(r)$ on $(0,\infty)$, we have for all $r>r_{0}$, 
\begin{align}\label{enuff}
 \frac{H(r, \sigma)}{2}-k\coth(r) &\geq \frac{\sqrt{\delta}}{2}\coth(\sqrt{\delta}\,r)-k\coth(r) 
\nonumber\\
&\geq k\coth(r_{0})\Big(\frac{\sqrt{\delta}}{2k\coth(r_{0})} - 1\Big)\geq\epsilon.
\end{align}
Expanding $\coth(t) = 1+O(e^{-2t})$ as a power series in $e^{-2t}$ and estimating error terms, the first line in \eqref{enuff} implies \eqref{intbound}. The last line in \eqref{enuff} combined with \eqref{smallr} implies the lower bound in \eqref{qbnd}. The upper bound is a simple consequence of \eqref{hbnd}.   
Finally we remark that when $\kappa=0$ (i.e., the sectional curvature is non-positive), one has the freedom of choosing any sufficiently large $r_{0}$ for which \eqref{r0} is satisfied. 
\end{proof}

We are now ready to prove the main technical result of this section. 

\begin{theorem}\label{Price1}
Let $(M^{n}, g)$ be a Riemannian manifold with $-1\leq\sec_{g}\leq \kappa$, $0\leq\kappa\leq1$ and dimension $n\geq 6$. Let $\alpha$ be a harmonic $k$-form. Assume
\[
-Ric \geq \delta g, \quad \delta> 4k^2.
\] 
 Let again $r_0=\frac{1}{\sqrt{\kappa}} arccot\big(\frac{\sqrt{\delta}}{\sqrt{\kappa}(n-1)}\big)$. Let 
$C_{\sigma,k}=\coth(\sigma)^{2k}.$
If for some $\epsilon>0$,
\begin{align}\label{assump17}\frac{\sqrt{\delta}}{2k}\geq \coth(r_0)+\frac{\epsilon}{k},  \end{align}
 then 
for any $0<\sigma<\tau\leq\gamma_{g}(M)$ and $p\in M$, 
\begin{align}\label{priceeq}
\int_{B_\sigma(p)}|\alpha|^2dv\leq \frac{(n-1)C_{\sigma,k}e^{-2(\tau-\sigma)(\frac{\sqrt{\delta}}{2}-k)}}{\epsilon(1-C_{\sigma,k}e^{-2(\tau-\sigma)(\frac{\sqrt{\delta}}{2}-k)})}\int_{B_\tau(p)\ssm B_\sigma(p)}|\alpha|^2dv.
\end{align}
\end{theorem}

\begin{proof}
By Lemma \ref{Riccati-Rauch1}, $q(r)\geq \epsilon>0$. Equation \eqref{lie4} then implies $ 1-2\mu(s)>0$ for $s\in [0,\gamma_g(M)]$. Then the weight function $\phi$ defined in \eqref{phidef} satisfies 
\begin{align}\phi(r)\leq e^{-2\int_\sigma^rq(s)ds}\leq C_{\sigma,k}e^{-2(r-\sigma)(\frac{\sqrt{\delta}}{2}-k)},
\end{align}
where we have used \eqref{intbound} for the last inequality, and we set 
$C_{\sigma,k}=\coth(\sigma)^{2k}$ (not sharp). Inserting this inequality into \eqref{price} yields
\begin{align}\label{preeq}
&(1-C_{\sigma,k}e^{-2(\tau-\sigma)(\frac{\sqrt{\delta}}{2}-k)})\int_{B_\sigma(p)}q(r)|\alpha|^2dv\nonumber\\
&\leq C_{\sigma,k}e^{-2(\tau-\sigma)(\frac{\sqrt{\delta}}{2}-k)}\int_{B_\tau(p)\ssm B_\sigma(p)}q(r)|\alpha|^2dv.
\end{align}
The pointwise bounds on $q(r)$ given in Lemma \ref{Riccati-Rauch1} now yield 
\begin{align}\label{preeq2}
&(1-C_{\sigma,k}e^{-2(\tau-\sigma)(\frac{\sqrt{\delta}}{2}-k)})\epsilon\int_{B_\sigma(p)} |\alpha|^2dv\nonumber\\
&\leq C_{\sigma,k}e^{-2(\tau-\sigma)(\frac{\sqrt{\delta}}{2}-k)}(n-1)\coth(\sigma)\int_{B_\tau(p)\ssm B_\sigma(p)}|\alpha|^2dv,
\end{align}
which then concludes the proof. Once again when $\kappa=0$, one has the freedom of choosing any sufficiently large $r_{0}$ for which \eqref{r0} is satisfied. 
\end{proof}

\section{From Integral Inequalities to Dimension Estimates}\label{Dimension}
In this section, we use the integral inequalities of the previous section to extract dimension estimates. In the following section, we will sharpen these estimates using cohomological techniques. We include the slightly weaker results here for the reader who may be interested in applying them to related problems - such as bounding the dimensions of eigenspaces with very small eigenvalue - to which these techniques, but not their cohomological improvement, readily apply.   

In order to extract dimension estimates from \eqref{priceeq}, we first need the following standard lemma. 
\begin{lemma}\label{Picking}
Let $(M^{n}, g)$ be a closed Riemannian manifold. There exists $\alpha\in\mathcal{H}^{k}_{g}(M)$, with $||\alpha||_{L^{2}}=1$,
such that 
\[
\max_{p\in M}|\alpha|^{2}\geq \frac{k!(n-k)!b_{k}(M)}{n!Vol(M)},
\]
where $b_{k}(M)=\dim_{\IR}\mathcal{H}^{k}_{g}(M)$ is the $k$-th Betti number of $M$.
\end{lemma}
\begin{proof}
Let $K(x,y)$ denote the Schwartz kernel of the $L^2-$orthogonal projection onto $\mathcal{H}^{k}_{g}(M)$. Then 
\begin{align}\int_M\tr K(x,x)dv = dim\mathcal{H}^{k}_{g}(M).\end{align}
Hence, there exists $p\in M$ such that 
\begin{align}\tr K(p,p)\geq \frac{b_k(M)}{Vol(M)}.\end{align}
Then there exists a unit eigenvector $z$ of $K(p,p)$ with eigenvalue 
\[
\lambda\geq \frac{k!(n-k)!b_k(M)}{n!Vol(M)}.
\]
Since $K$ is  the Schwartz kernel of an $L^2-$orthogonal projection,
\begin{align}\|K(x,p)z\|_{L^2}^2 = \langle K(p,p)z,z\rangle =\lambda.
\end{align}
Set 
\begin{align}\alpha(x):= \frac{K(x,p)z}{\sqrt{\lambda}}.
\end{align}
Then $\|\alpha\|_{L^2}=1$, and 
\begin{align}|\alpha(p)|^2 = \frac{|K(p,p)z|^2}{\lambda}= \lambda\geq \frac{k!(n-k)!b_k(M)}{n!Vol(M)}.
\end{align}
\end{proof}

The following lemma allows us to pass from  integral inequalities to pointwise estimates.  
\begin{lemma}\label{Moser}
Let $(M^{n}, g)$ be a closed Riemannian manifold with 
\[
-1\leq\sec_{g}\leq 1.
\]
Given a harmonic $k$-form $\alpha\in\mathcal{H}^{k}_{g}(M)$, for any $p\in M$ and $R<\min(\gamma_{g}(M), 1)$ there exists a strictly positive constant $d(n, R):= d(n)(1+\frac{1}{R})^n$ such that
\[
||\alpha||^{2}_{L^{\infty}(B_{\frac{R}{2}}(p))}\leq d(n, R)||\alpha||^{2}_{L^{2}(B_{R}(p))}.
\]
\end{lemma}
\begin{proof}
This proof is a standard application of Moser iteration. See for example \cite[Proposition 2.2]{LS}, where the theorem is proved for hyperbolic manifolds. The extension to our context follows immediately substituting the bounded variable curvature for constant curvature and using  Lemma 2.24 of {\cite[Chapter 2]{aubin}}  to replace \cite[(2.1)]{LS}. More precisely, one replaces \cite[(2.1)]{LS} by the statement that there exists 
$ S_M>0$, 
such that for all $ p\in M$, $ R<\max\{\frac{\delta(M)}{2},1\} $, and all compactly supported smooth functions on the ball $ \xi\in C_c^\infty(B_{R}(p))$,  one has 
\begin{align}\label{sob3}S_M \|d \xi\|^2_{L^2(B_R(p))}\geq \| \xi\|_{L^{\frac{2n}{n-2}}(B_R(p))}^2. 
\end{align}
We refer to {\cite[Chapter 2]{aubin}} for this well-known statement.\\
\end{proof}

We close this section with our main application.

\begin{theorem}\label{application1}
Let $(M^{n}, g)$ be a Riemannian manifold with $-1\leq\sec_{g}\leq \kappa$. Assume
\[
-Ric \geq \delta g, \quad \delta\geq 4k^2.
\] 
Let $r_0=\frac{1}{\sqrt{\kappa}} arccot(\frac{\sqrt{\delta}}{\sqrt{\kappa}(n-1)})$. If
\begin{align}\label{assump3}\frac{\sqrt{\delta}}{2k}\geq \coth(r_0)+\frac{\epsilon}{k},  \end{align}
 then there exists a positive constant $c(n)$ depending on the dimension only such that for $\gamma_g(M)$ large
\[
\frac{b_{k}(M)}{ Vol_{g}(M)}\leq c(n)\epsilon^{-1}e^{-2\gamma_g(M)(\frac{\sqrt{\delta}}{2}-k)}.
\] 
\end{theorem}
\begin{proof}
By Lemma \ref{Picking}, there exists   $\alpha\in\mathcal{H}^{k}_{g}(M)$ and $p\in M$, with 
$||\alpha||_{L^{2}(M)}=1$, and
\begin{align}\label{1}
 |\alpha(p)|^{2}\geq\frac{k!(n-k)!b_k(M)}{n!Vol(M)}.
\end{align}
 On the other hand,  Lemma \ref{Moser} and equation \eqref{priceeq} yield
\begin{align}\label{2}
|\alpha|^{2}(p)\leq d(n, 1/2)||\alpha||^{2}_{L^{2}(B_{p}(1))} \leq \frac{\tilde c(n)}{\epsilon}e^{-2(\gamma_g(M)-1)(\frac{\sqrt{\delta}}{2}-k)}||\alpha||^{2}_{L^{2}(B_{p}(\tau))},
\end{align}
 where $d(n, 1/2)>0$ is the constant of Lemma \ref{Moser}, and
\[
\tilde c(n):=2d(n, 1/2)(n-1)C_{1,k+\frac{1}{2}}, 
\]
with $\gamma_g(M)$  large enough so that 
\[
1 - C_{1,k}e^{-2(\gamma_g(M)-1)(\frac{\sqrt{\delta}}{2} -k)}\geq\frac{1}{2}.
\]
\end{proof}

\section{Excising Geodesic Balls}

In Section \ref{Ricci Bound}, we estimate $q(r)$ from below by combining Rauch's comparison with a Riccati type argument. Under the curvature constraint $\frac{\sqrt{\delta}}{2}>k\coth(r_0),$ with $r_0:= \frac{1}{\sqrt{\kappa}}arccot\big(\frac{\sqrt{\delta}}{\sqrt{\kappa}(n-1)}\big)$, these arguments suffice to provide a positive uniform lower bound on $q(r)$ for any $r$. The deleterious effects of positive curvature in the comparison arguments diminish at large radius. In fact, it only affects Rauch's comparison for small values of $r$. On the other hand, under our usual Ricci curvature assumption, the positive curvature does not affect the Riccati argument for large values of $r$. Thus it is natural to work in the complement of a large ball. In this section, we use cohomological arguments to remove the application of Rauch's comparison for small values of $r$. 
We begin with the following reformulation of  Equation \eqref{lie4} with two boundary components.

\begin{proposition}\label{Bfundamental0}
Let $(M^{n}, g)$ be complete and satisfy  $-1\leq \sec_{g}$. Given a point $p\in M$, $\alpha\in\mathcal{H}^{k}_{g}(M)$, and $\sigma\leq \tau\leq \gamma_g(M)$, we have
\begin{align}\label{dble01}\notag
\int_{S_{\tau}(p)}\Big(\frac{1}{2}-\mu(\tau)\Big)|\alpha|^{2}d\sigma &= \int_{S_{\sigma}(p)}\Big(\frac{1}{2}-\mu(\sigma)\Big)|\alpha|^{2}d\sigma\\ 
&+\int_{B_{\tau}(p)\ssm B_{\sigma}(p)}q(r)|\alpha|^{2}dv.
\end{align}
Moreover, with $\phi$ as defined in \eqref{phidef}, 
\begin{align}\label{dble2}\phi(\tau)\int_{B_{\tau}(p)\ssm B_{\sigma}(p)}q(r)|\alpha|^2dv = (1-\phi(\tau))\int_{S_{\sigma}(p)}\Big(\frac{1}{2}-\mu(\sigma)\Big)|\alpha|^2d\sigma.
\end{align}
\end{proposition}
\begin{proof}
This follows from the same arguments as \eqref{lie4} and \eqref{price}. 
\end{proof}
 
In Section \ref{Dimension}, we showed that if $\phi$ decays to zero, it is possible to give effective estimates on the size of Betti numbers normalized by the Riemannian volume. In order to control the size of $\phi$, it is necessary to understand the sign not only of $q(r)$ but also of $\frac{1}{2}-\mu(r)$, see Equation \eqref{phidef}. In particular, we require $\mu(r)<\frac{1}{2}$ for all $r$ sufficiently large. This inequality follows from \eqref{dble01}, if  $\mu(\sigma)<\frac{1}{2}$. Thus, it is natural to work with harmonic forms with Neumann boundary condition. 
 
  Given a complete Riemannian manifold with boundary, $\Omega$, let $\mathcal{H}^{k}_{2,N}(\Omega)$ denote the $L^2$-harmonic  $k$-forms on $\Omega$ satisfying Neumann boundary conditions on $\p\Omega$. Let $H^{k}_{2,N}(\Omega)$ denote the (absolute) $L^2-$ cohomology of $\Omega$. This is the cohomology of the complex $(C_2^{\ast},d)$, where 
$C_2^k$ denotes the closure in the $L^2$ graph norm ($\|\phi\|^2_{\text{graph}}:= \|\phi\|^2_{L^2}+ \|d\phi^2\|_{L^2}$) of the smooth $L^2$ $k-$ forms on $\Omega$. 
 
It is easy to see that, if $0$ is not in the essential spectrum of $\Delta_k$, then
\begin{align}\label{gap}
\mathcal{H}^{k}_{2,N}(\Omega)\simeq H^{k}_{2, N}(\Omega).
\end{align} 
When $\Omega$ is compact (possibly with boundary), then elliptic regularity plus the de-Rham isomorphism implies 
$ H^{k}_{2, N}(\Omega)\simeq H^{k}(\Omega).$
For more details on $L^{2}$-cohomology we refer to Section \ref{Atiyah}.

The next lemma is the natural analogue of Theorem \ref{Price1} for Neumann harmonic forms.
 
\begin{lemma}\label{A}
Let $(M^{n}, g)$ be a complete Riemannian manifold with $-1\leq \sec_{g}\leq 1$. 
Let $p\in M$, $\rho<\gamma_g(M)$, and  $\alpha\in \mathcal{H}^{k}_{2,N}(M\ssm B_\rho(p))$. Assume there exists $\delta>4k^2$ and $\epsilon>0$ such that
\[
-Ric \geq \delta g,
\]
and
\begin{align}\label{qpos}
\frac{\sqrt{\delta}}{2}\coth(\sqrt{\delta}\rho)-k\coth(\rho)\geq \epsilon.
\end{align}
Then for any $\sigma\in( \rho,\gamma_g(M))$, $\tau\in( \sigma,\gamma_g(M))$ and $\sigma<R<\tau$, 
\begin{align}\label{dble29} \int_{B_{R}(p)\ssm B_{\sigma}(p)}|\alpha|^2dv\leq   \frac{n-1}{\epsilon}\coth(\sigma)^{2k+1}e^{-(\sqrt{\delta}-2k)(\tau-R)}\int_{B_{\tau}(p)\ssm B_{\sigma}(p)}|\alpha|^2dv.
\end{align}
\end{lemma}
\begin{proof}
Recall that the Neumann boundary condition implies $i_{\partial_{r}}\alpha=0$ on $S_\rho(p).$ Equivalently, 
$\mu(\rho) = 0$.  Hence  \eqref{dble01} with $\sigma=\rho$ implies $\mu(r) \leq \frac{1}{2}$, $\forall r\in [\rho,\gamma_g(M)]$, as long as $q(r)\geq 0$, for $\forall r\in [\rho,\gamma_g(M)]$. 
Taking \eqref{dble2} for two different values of $\tau$ and then taking the difference of the two equations gives 
\begin{align}\label{dblediff}&\phi(\tau)\int_{B_{\tau}(p)\ssm B_{\sigma}(p)}q(r)|\alpha|^2dv  \nonumber\\
= \phi(R)\int_{B_{R}(p)\ssm B_{\sigma}(p)}q(r)|\alpha|^2dv&+ (\phi(R)-\phi(\tau))\int_{S_{\sigma}(p)}\Big(\frac{1}{2}-\mu(\sigma)\Big)|\alpha|^2d\sigma.
\end{align}
The hypotheses imply $\phi$ is monotonically decreasing. Our curvature estimates imply $q(s)> \frac{\sqrt{\delta}}{2}-k\coth(s),$ and the desired estimate follows from \eqref{phidef}. 
\end{proof}

The Price inequality of Lemma \ref{A} can now be applied to derive a vanishing for certain spaces of Neumann harmonic forms. 

\begin{corollary}
Let $(M^{n}, g)$ be a simply connected non-compact complete Riemannian manifold without conjugate points and $-1\leq \sec_{g}\leq 1$. Assume there exists $\delta>4k^2>0$ such that
\[
-Ric \geq \delta g.
\]
Let $p\in M$. Then for $\rho$ sufficiently large, $\mathcal{H}^{k}_{2, N}(M\ssm B_{\rho}(p))=0$. 
\end{corollary}

\begin{proof}
Taking $\tau\to\infty$ in the preceding lemma, we see that $\alpha$ vanishes identically.
\end{proof}

Next, we consider closed manifolds without conjugate points.

\begin{proposition}\label{-1ball}
Let $(M^{n}, g)$ be a closed manifold without conjugate points with $-1\leq \sec_{g}\leq 1$. 
Assume there exists $\delta>4k^2$ such that
\[
-Ric \geq \delta g.
\]
Let $\pi_i:M_i\to M$ be a sequence of  Riemannian covers of $M$ with $\gamma_g(M_i)\to \infty$.  Let $h_i\in \mathcal{H}^{k}(M_i ),$ and $p_i\in M_i$. 
Then there exists $c(n, k, \delta)>0$ so that for $\gamma_{g_{i}}(M_i)$ sufficiently large,    
\begin{align}\label{ndble29} \int_{B_\rho(p_i)}|h_i|^2dv\leq   c(n, k, \delta)e^{-(\sqrt{\delta}-2k)\gamma_g(M_i)}\|h_i\|^2.
\end{align}
\end{proposition}

\begin{proof} Choose $\rho$ sufficiently large so that 
\[
\frac{\sqrt{\delta}}{2}\coth(\sqrt{\delta}\rho)-k\coth(\rho) \geq \frac{\sqrt{\delta}-2k}{4}=:\epsilon>0.
\] Consider  $i$ sufficiently large so that $\rho<\gamma_g(M_i)$. 
Then for $k<n-1 $,    
we have 
$$\mathcal{H}^{k}(M_i )\simeq H^k(M_i) \simeq H^k_N(M_i\ssm B_\rho(p_i))\simeq \mathcal{H}^{k}_{N}(M_i\ssm B_{\rho}(p_i)).$$ 
Consider the map from $\mathcal{H}^{k}(M_i)$ obtained as follows. Given $h\in \mathcal{H}^{k}(M_i )$, let $J_1(h)$ denote the $H^k_{N}(M_i\ssm B_\rho(p_i))$ harmonic representative of the restriction of $h$ to $M_i\ssm B_\rho(p_i)$. Let $b$ denote a coexact primitive (with respect to the complex $(C_2^\ast(B_{\rho+2}(p_i)\ssm B_\rho(p_i)),d)$)  for the restriction of $J_1(h)$ to $B_{\rho+2}(p_i)\ssm B_\rho(p_i)$. Let $\eta$ be a smooth cutoff function with $|d\eta|<2$, supported on $B_{\rho+2}(p_i)\ssm B_\rho(p_i)$, identically zero near $\p B_{\rho+2}(p_i)$ and identically one in $B_{\rho+1}(p_i)$. 
Then $J_2(h):= J_1(h)-d(\eta b)$ defines an element in $H^k(M_i)$. Because every cycle in $H_k(M_i)$ has a representative disjoint from $B_\rho(p_i)$ and the integral of $J_2(h)$ over every such cycle is equal to the integral of $h$ over the cycle, $J_2(h)$ is cohomologous to $h$. Let now $J_3(h)$ denote the $M_i$  harmonic projection of $J_{2}(h)$. Then we have shown $J_3(h) = h$. Restriction and harmonic projection are norm nonincreasing. In particular, we have 
\begin{align}\|J_1(h)\|^2_{L^2(M_i\ssm B_\rho(p_i))}\leq \|h\|_{L^2(M_i)}^2 - \int_{B_\rho}|h|^2dv.
\end{align}
On the other hand, we have 
\begin{align}\|h\|^2_{L^2(M_i)}&=\|J_3(h)\|^2_{L^2(M_i)}\leq \|J_2(h)\|^2_{L^2(M_i)}\nonumber\\
&\leq \|h\|_{L^2(M_i)}^2 - \int_{B_\rho}|h|^2dv +\|d(\eta b)\|^2_{L^2(B_{\rho+2}(p_i)\ssm B_\rho(p_i))}.
\end{align}
Hence 
\begin{align}\label{neumest} \notag
\int_{B_\rho}|h|^2dv &\leq\|d(\eta b)\|^2_{L^2(B_{\rho+2}(p_i)\ssm B_\rho(p_i))}\nonumber\\
&\leq 2\| \eta J_1(h)\|^2_{L^2(M_i\ssm B_\rho(p_i))} + 2\|d \eta\wedge b\|^2_{L^2(B_{\rho+2}(p_i)\ssm B_\rho(p_i))} \\ 
&\leq 2||J_{1}(h)||^{2}_{L^{2}(B_{\rho+2}\ssm B_{\rho})}+8||b||^{2}_{L^{2}(B_{\rho+2}\ssm B_{\rho})}.
\end{align}
In order to bound $||b||^{2}_{L^{2}(B_{\rho+2}\ssm B_{\rho})}$, it suffices to estimate the $L^2$ norm of any primitive for the restriction of $J_1(h)$ to the annular region, since the coexact primitive has smallest $L^2$ norm among all primitives. We now construct one such primitive.
Use normal coordinates to fix a diffeomorphism $\zeta:B_{\rho+2}(p_i)\setminus B_\rho(p_i)\to S^{n-1}\times [0,2]$. 
Write $H:=\zeta^*J_1(h) = H_0+dt\wedge H_1$, with $i_{\frac{\p}{\p t}}H_j = 0$, $t$ the coordinate on $[0,2]$.  Fix $L\in [0,2]$ minimizing $\int_{S_L}|H|^2d\sigma.$ Set 
\begin{align}\beta_L(r) := \int_L^rH_1(s)ds.
\end{align}
Then since $d(\zeta^*J_1(h))=0$
$$d\beta_L = H - H_0(L).$$
In the product metric $g_0$ on $S^{n-1}\times [0,2]$, we have 
\begin{align}
\|\beta_{L}\|^2_{L^2(S^{n-1}\times [0,2],g_0)}=&\int_0^2\int_{S^{n-1}}\Big(\int_L^rH_1(s)ds\Big)^2d\sigma dr\leq 2 \int_{S^{n-1}\times [0,2]}|H_1|^2dv\\ \notag
& \leq 2\|H\|^{2}_{L^{2}(S^{n-1}\times [0, 2], g_0)}.
\end{align}
(We now indicate the metric in our $L^2$ norms when confusion may occur.)
Let $\xi$ denote a coexact primitive for $H_0(L)$ viewed as an exact form on $S^{n-1}$. 
Then in the product metric,
$$\|\xi\|^2_{L^2(S^{n-1}\times [0,2],g_0)}\leq W_n\|H\|^2_{L^2(S^{n-1}\times [0,2],g_0)},$$
with $W_n= O(\lambda_{1,k}^{-1}),$ where $\lambda_{1,k}$ denotes the first eigenvalue of the Laplace Beltrami operator for $k$-forms on the sphere. (Our hypotheses imply $k\not = 0,n-1.)$ By construction, 
$$d(\beta_L+\xi) = H.$$
There exists $C_1(\rho)>0$ so that 
\begin{align}\label{quasi-isometric}
C_1^{-1}\|\cdot\|^2_{L^2(B_{\rho+2}(p)\ssm B_\rho(p), g)}\leq \|\cdot\|^2_{L^2(S^{n-1}\times [0,2], g_0)}\leq C_1\|\cdot\|^2_{L^2(B_{\rho+2}(p)\ssm B_\rho(p), g)}.
\end{align}
Then we have 
\begin{align}\label{primest}
\|\beta_L+\xi\|^2_{L^2(S^{n-1}\times [0, 2], g_0)}\leq (W_{n}+2\sqrt{2}\sqrt{W_{n}}+2)\|H\|^2_{L^2(S^{n-1}\times [0, 2], g_0)}. 
\end{align}
Equation \eqref{primest} combined with \eqref{neumest} and \eqref{quasi-isometric} gives 
\begin{align}\label{almest} 
\int_{B_\rho}|h|^2dv \leq  d(W_{n}, C_{1}) \|J_1(h)\|^2_{L^2(B_{\rho+2}(p)\ssm B_\rho(p),g)},
\end{align}
where $d(W_{n}, C_{1})$ is a positive constant depending on $W_{n}$ and $C_{1}$.
Using \eqref{dble29}, we have 
\begin{align}
\int_{B_\rho}|h|^2dv \leq   d(W_{n}, C_{1})  \frac{n-1}{\epsilon}\coth(\rho)^{2k+1}e^{-(\sqrt{\delta}-2k)(\tau-\rho-2)}\int_{B_\tau\ssm B_{\rho}}|J_1(h)|^2dv,
\end{align}
and the result follows. 
\end{proof}

We can now prove the main result of this section.

\begin{corollary}\label{intro2}
Let $(M^{n}, g)$ be a closed Riemannian manifold without conjugate points and with $-1\leq \sec_{g}\leq 1$. 
Assume there exists $\delta>4k^2>0$ such that
\[
-Ric \geq \delta g.
\]
Let $\pi_i:M_i\to M$ be a sequence of Riemannian covers of $M$ with $\gamma_g(M_i)\to \infty$. 
Then there exists $b(n,k)>0$ so that for $\gamma_g(M_i)$ sufficiently large,    
\begin{align}\label{bdble29} 
\frac{b_k(M_i)}{Vol(M_i)}\leq   b(n, k, \delta)e^{-(\sqrt{\delta}-2k)\gamma_{g_{i}}(M_i)}.
\end{align}
In particular, 
\begin{align}\label{limble29} \lim_{i\to\infty}\frac{b_k(M_i)}{Vol(M_i)}=0.
\end{align}
\end{corollary}
\begin{proof} 
Apply Lemmas \ref{Picking} and \ref{Moser} to Proposition \ref{-1ball}. 
\end{proof} 
Congruence subgroups of arithmetic groups of $\IQ$ rank $0$ algebraic groups provide an important and widely studied class of examples of towers of covers. These lattices have large injectivity radii relative to their covolumes with respect to the natural locally symmetric metric. For more details see \cite{Sarnak}, \cite{Yeung1}, and \cite{Marshall}. The following definition abstracts the injectivity radius properties of such lattices. 

\begin{define}\label{congruence type}
Let $(M^{n}, g)$ be a closed Riemannian manifold with infinite residually finite fundamental group $\Gamma:=\pi_{1}(M^{n})$. Given a cofinal filtration $\{\Gamma_{i}\}$ of $\Gamma$, for any index $i$ let us set
\[
r_{i}:=\inf\{d_{\tilde{g}}(z, \gamma_{i}z)/2\quad |\quad z\in\tilde{X}, \quad \gamma_{i}\in\Gamma_{i}, \quad \gamma_{i}\neq 1\}
\]
where $(\tilde{X}, \tilde{g})$ is the Riemannian universal cover. We say that the cofinal filtration is of \textbf{congruence type} if there exist constants, $0<\alpha(n, \{\Gamma_{i}\})<1$, $g(n, \{\Gamma_{i}\})>0$ such that
\[
e^{r_{i}}\geq g(n, \{\Gamma_{i}\})[\Gamma:\Gamma_{i}]^{\alpha(n, \{\Gamma_{i}\})}
\]
for any $i\geq 1$, where $[\Gamma:\Gamma_{i}]$ is the index of $\Gamma_{i}$ in $\Gamma$. We call the constant $\alpha$  the \textbf{exponent} of the cofinal filtration.
\end{define}

We restate our Betti number asymptotics under the additional hypothesis of a congruence type cofinal filtration with exponent $\alpha$. 

\begin{corollary}\label{Betti1}
Let $(M^{n}, g)$ be a Riemannian manifold without conjugate points with $-1\leq\sec_{g}\leq \kappa$ and residually finite fundamental group $\Gamma$. Assume
\[
-Ric \geq \delta g, \quad \delta> 4k^2.
\] 
Let $\{\Gamma_{i}\}_i$ be a congruence type cofinal filtration of $\Gamma$ of exponent $\alpha$. Denote by $\pi_{i}: M_{i}\rightarrow M$ the regular Riemannian cover of $M$ associated to $\Gamma_{i}$. Then
\[
b_{k}(M_{i})\leq d(n,k, \delta, \Gamma) Vol(M_{i})^{1-2 \alpha(\frac{\sqrt{\delta}}{2}-k)}.
\]
where $d(n, k,\Gamma)$ is a positive constant. 
\end{corollary}

\begin{proof}

Since $(M^{n}, g)$ has no conjugate points, for any $i\geq 1$ the numerical invariant $r_{i}>0$ given in Definition \ref{congruence type} is simply the injectivity radius $\gamma_{g_{i}}(M_{i})$. Since
\[
Vol(M_{i})=Vol(M)[\Gamma:\Gamma_{i}],
\]
by Corollary \ref{intro2} we have
\[
b_{k}(M_{i})\leq\frac{b(n,k)}{e^{2r_i(\frac{\sqrt{\delta}}{2}-k)}}Vol(M_i)\leq\frac{b(n,k)Vol(M)}{g(n,\Gamma)^{2(\frac{\sqrt{\delta}}{2}-k)}}[\Gamma:\Gamma_{i}]^{1-2(\frac{\sqrt{\delta}}{2}-k)\alpha},
\]
and the proof is complete.
\end{proof}

\section{Inequalities for Negatively Pinched Manifolds in Dimensions $n\geq 4$}\label{negative pinched}

Let $(M^{n}, g)$ be a closed Riemannian manifold of dimension $n\geq 4$ such that
\begin{align}\label{curvass}
-b^{2}\leq \sec_{g}\leq -a^{2},
\end{align}
for some  $a, b\in (0,\infty)$.  
Let $p\in M$, and denote by $B_{R}$ the ball of radius $0<R<\gamma_g(M)$ centered at $p$.  In $B_R$ we write the metric in geodesic spherical coordinates as $g=dr^{2}+g_{r}$. With the curvature assumptions \eqref{curvass}, Rauch's comparison (see Theorem \ref{comparison}) 
gives the following  two sided bound on the second fundamental form of any geodesic sphere:
\begin{align}\label{2nd}
a\coth(ar)g_r(u)\leq \textrm{Hess}(r, u) \leq b\coth(br)g_r(u),
\end{align}
for any $r\leq\gamma_{g}(M)$ and for any point $u$ on the geodesic sphere $S_{r}$. Taking the trace of \eqref{2nd} gives the corresponding mean curvature bound:
\begin{align}\label{3rd}
(n-1)a\coth(ar)\leq H(r, u) \leq (n-1)b\coth(br),
\end{align}
for any $u\in S_{r}(p)$, given any $p\in M$. This implies that for any $r\leq \gamma_{g}(M)$, we have
\begin{align}\label{qpinch}
\frac{(n-1)}{2}a\coth(ar)-kb\coth(br)\leq q(r)\leq \frac{(n -1)}{2}b\coth(br)-ka\coth(ar).
\end{align}

We now specialize our Price inequalities to harmonic forms on closed pinched negatively curved manifolds of dimension $n\geq 4$. This is the main result of this section.\\ 
 
\begin{theorem}\label{Price2}
Let $(M^{n}, g)$ be a compact Riemannian manifold of dimension $n\geq 4$. Assume the sectional curvature is $\epsilon$-pinched :
\[
-(1+\epsilon)^{2}\leq\sec_{g}\leq -1,
\] 
with $\epsilon\geq 0$. Let $k$ be a non-negative integer such that
\[
\epsilon_{n, k}:=(n-1)-2k(1+\epsilon)> 0.
\]
For $\alpha\in   \mathcal{H}^{k}_{g}(M)$,
and $0<\sigma<\tau\leq\gamma_{g}(M)$, we have
\begin{align}\label{ovenbaked}    \int_{B_{\sigma}(p)} |\alpha|^2 dv\leq \frac{e^{- (\tau-\sigma)\epsilon_{n, k}}}{1-e^{- (\tau-\sigma)\epsilon_{n, k}}} \frac{(n-1)}{\epsilon_{n,k}}(1+\epsilon)\coth(\sigma)\int_{B_{\tau}(p)\ssm B_\sigma(p)} |\alpha|^2 dv.
\end{align}
\end{theorem}

\begin{proof}
Recall the Price equality given in \eqref{price}
\begin{align} \int_{B_{\sigma}(p)}q(r)|\alpha|^2 dv= \phi(\tau) \int_{B_{\tau}(p)}q(r)|\alpha|^2 dv.  
\end{align}
By \eqref{qpinch} we have 
\begin{align}\label{epsilonpinch}q(r)\geq \frac{(n-1)}{2}\coth(r)-k(1+\epsilon)\coth((1+\epsilon)r)>\frac{(n-1)}{2} -k(1+\epsilon) .
\end{align}
Hence by \eqref{phidef}, we have
\begin{align}\phi(r)\leq e^{- (r-\sigma)( n-1  -2k(1+\epsilon))},
\end{align}
which combined with \eqref{price} gives 
\begin{align}\label{semi} \int_{B_{\sigma}(p)}q(r)|\alpha|^2 dv\leq \frac{e^{- (r-\sigma)( n-1  -2k(1+\epsilon))}}{1-e^{- (r-\sigma)( n-1  -2k(1+\epsilon))}} \int_{B_{\tau}(p)\ssm B_\sigma(p)}q(r)|\alpha|^2 dv.  
\end{align}
The upper bound of Equation \eqref{qpinch} implies
 \[
q(r)\leq \frac{(n-1)}{2}(1+\epsilon)\coth((1+\epsilon)r)-k\coth(r),
\]
so that for any $r\geq\sigma$, we have
\begin{align}\label{epsilonpinchup} 
q(r)\leq \frac{n-1}{2}(1+\epsilon)\coth(r).
\end{align}
The lower bound on $q$ given in \eqref{qpinch} combined with \eqref{epsilonpinchup} then yields
\begin{align}\label{baked} \int_{B_{\sigma}(p)} |\alpha|^2 dv\leq \frac{e^{- (\tau-\sigma)\epsilon_{n, k}}}{1-e^{- (\tau-\sigma)\epsilon_{n, k}}} \frac{(n-1)}{\epsilon_{n, k}}(1+\epsilon)\coth(\sigma)\int_{B_{\tau}(p)\ssm B_\sigma(p)} |\alpha|^2 dv,  
\end{align}
as claimed. 
\end{proof}
 
\begin{corollary}\label{firstattempt}
Let $(M^{n}, g)$ be a compact Riemannian manifold of dimension $n\geq 4$. Assume 
\[
-(1+\epsilon)^{2}\leq \sec_{g}\leq -1 
\]
with $\epsilon\geq 0$. There exists a constant $c(n,k)>0$ so that for each  non-negative integer $k$ such that
\[
\epsilon_{n, k}:=(n-1)-2k(1+\epsilon)> 0,
\]
\[
 \frac{b_{k}(M)}{Vol_{g}(M)}\leq c(n, k)e^{- \epsilon_{n, k} \gamma_{g}(M)},
\]  
for $\gamma_g(M)>1+\frac{\ln(2)}{\epsilon_{n,k}}.$
\end{corollary}
\begin{proof}
Apply Lemmas \ref{Picking} and \ref{Moser} to Proposition \ref{Price2}. The (nonsharp) constraint on $\gamma_g(M)$ serves to control denominators. 
\end{proof}

Finally, we study the borderline case where $\epsilon_{n, k}=0$ and $\epsilon>0$.  

\begin{theorem}\label{Price3}
Let $(M^{n}, g)$ be a compact Riemannian manifold of dimension $n\geq 4$. Assume the sectional curvature is $\epsilon$-pinched :
\[
-(1+\epsilon)^{2}\leq\sec_{g}\leq -1,
\] 
with $\epsilon>0$. Let $k$ be a positive integer such that
\[
\epsilon_{n, k}:=(n-1)-2k(1+\epsilon)=0.
\]
For $\alpha\in  \mathcal{H}^{k}_{g}(M)$
and $0<1\leq\tau\leq\gamma_{g}(M)$, we have
\begin{align} \label{q3}
\int_{B_{1}(p)} |\alpha|^2 dv\leq \frac{\sinh^2(1+\epsilon )}{2k(1+\epsilon)\epsilon (\tau-1)}\int_{B_{\tau}(p)\ssm B_{1}(p)}|\alpha|^{2}dv.
\end{align}
\end{theorem}

\begin{proof}
Given the assumption that $\epsilon_{n,k}=0$, we write the first inequality of \eqref{epsilonpinch} as   
$$q(r) \geq  k(1+\epsilon)(\coth(r)-\coth((1+\epsilon)r).$$
Using the mean value theorem, we may estimate the lower bound by 
\begin{align}q(r) \geq  \frac{k(1+\epsilon)r\epsilon}{\sinh^2(r+\epsilon r)}.\end{align}
The righthand side of this inequality is monotonically decreasing in $r$. Hence 
for $r\in [0, 1]$ we have
\begin{align}\label{q2}
q(r) \geq \frac{k(1+\epsilon)\epsilon}{\sinh^2(1+\epsilon )}.
\end{align}
Recall  \eqref{price} and \eqref{lie4}: for $1<\tau\leq\gamma_{g}(M)$,
\begin{align} \int_{B_{\sigma}(p)}q(r)|\alpha|^2 dv= \phi(\tau) \int_{B_{\tau}(p)}q(r)|\alpha|^2 dv= \phi(\tau) \int_{S_{\tau}(p)}\Big(\frac{1}{2}-\mu(\tau)\Big)|\alpha|^2 d\sigma.  
\end{align}
When $\epsilon_{n,k} = 0$, $\phi$ need not decay exponentially, but by definition \eqref{phidef},
\begin{align}
0<\phi(r)\leq 1.
\end{align}

Thus, for any $s\geq 1$
\begin{align}\label{another}
\frac{k(1+\epsilon)\epsilon}{\sinh^2(1+\epsilon )}\int_{B_{1}(p)}|\alpha|^{2}dv&\leq\int_{B_{1}(p)}q(r)|\alpha|^{2}dv \notag \\ 
&\leq 1\cdot\int_{B_{s}(p)}q(r)|\alpha|^{2}dv\nonumber \\ 
&=\int_{S_{s}(p)}\Big(\frac{1}{2}-\mu(s)\Big)|\alpha|^{2}d\sigma.
\end{align}
Integrating Equation \eqref{another} from $1$ to $1<\tau\leq \gamma_{g}(M)$, yields
\[
\frac{k(1+\epsilon)\epsilon}{\sinh^2(1+\epsilon )}(\tau-1)\int_{B_{1}(p)}|\alpha|^{2}dv\leq\frac{1}{2}\int_{B_{\tau}(p)\ssm B_{1}(p)}|\alpha|^{2}dv,
\]
which concludes the proof.
\end{proof}
\begin{corollary}\label{2attempt}
Let $(M^{n}, g)$ be a closed Riemannian manifold of dimension $n\geq 4$. Assume 
\[
-(1+\epsilon)^{2}\leq \sec_{g}\leq -1, 
\]
with $\epsilon> 0$ satisfying  $k:= \frac{n-1}{2(1+\epsilon)}$ is an integer.
If $\gamma_{g}(M)>1$, 
there exists a positive constant $c(n, k, \epsilon)$ such that
\[
\frac{b_{k}(M)}{Vol_{g}(M)}\leq \frac{c(n, k, \epsilon)}{\gamma_{g}(M)-1}.
\]  
\end{corollary}

\begin{proof}
Choose a normalized $\alpha\in\mathcal{H}^{k}_{g}(M)$ as in Lemma \ref{Picking}, and let $p\in M$ be a point where the norm of $\alpha$ achieves its maximum. Since $||\alpha||_{L^{2}}=1$, for $1<\tau\leq \gamma_{g}(M)$, Equation \eqref{q3} combined with Lemma \ref{Picking} and with Lemma \ref{Moser} with $R=1$ then gives
\begin{align}\notag
\frac{b_{k}(M)}{Vol(M)}&\leq \binom{n}{k}|\alpha|^{2}_{p}\leq \binom{n}{k}\frac{d(n, 1)\sinh^2(1+\epsilon )}{2k(1+\epsilon)\epsilon (\tau-1)}||\alpha||^{2}_{L^{2}(B_{1})}\\ \notag
&\leq \binom{n}{k}\frac{d(n, 1)\sinh^2(1+\epsilon )}{2k(1+\epsilon)\epsilon (\tau-1)}.
\end{align}
Take $\tau=\gamma_{g}(M)$ to complete the proof. 
\end{proof}

We can now study the asymptotic behavior of Betti numbers with respect to coverings of negatively pinched Riemannian manifolds.

\begin{corollary}\label{Pinched-vanishing}
Let $(M^{n}, g)$ be a closed Riemannian manifold of dimension $n\geq 4$ with residually finite fundamental group $\Gamma:=\pi_{1}(M)$. Assume 
\[
-(1+\epsilon)^{2}\leq \sec_{g}\leq -1, 
\]
with $\epsilon\geq 0$. Given a cofinal filtration $\{\Gamma_{i}\}$ of $\Gamma$, denote by $\pi_{i}: M_{i}\rightarrow M$ the regular Riemannian cover of $M$ associated to $\Gamma_{i}$. For each non-negative integer $k$ such that 
$\epsilon_{n, k}=(n-1)-2k(1+\epsilon)> 0$,
\begin{align}\label{limit}
\lim_{i\rightarrow\infty}\frac{b_{k}(M_{i})}{Vol(M_{i})}=0.
\end{align}
Moreover, if $\epsilon>0$, the same holds true for $k$-forms of critical degree such that $\epsilon_{n, k}=0$.
\end{corollary}

\begin{proof}
Given a cofinal filtration $\{\Gamma_{i}\}$ of $\Gamma$, 
\[
\lim_{i\rightarrow\infty}\gamma_{g_{i}}(M_{i})=\infty
\]
where $M_{i}$ is equipped with the pull back metric $g_{i}:=\pi^{*}_{i}(g)$. For a proof, see Theorem 2.1 in \cite{Wallach}. The corollary is now a consequence of Corollary \ref{firstattempt} if  $\epsilon_{n, k}>0$. If $\epsilon_{n, k}=0$, we instead appeal to Corollary \ref{2attempt}.
\end{proof}

We now consider cofinal filtrations of congruence type (see Definition \ref{congruence type}). 

\begin{corollary}\label{pinchedalpha}
Let $(M^{n}, g)$ be a closed Riemannian manifold of dimension $n\geq 4$ with residually finite fundamental group $\Gamma:=\pi_{1}(M)$. Assume 
\[
-(1+\epsilon)^{2}\leq \sec_{g}\leq -1, 
\]
with $\epsilon\geq 0$. 
Given a cofinal filtration $\{\Gamma_{i}\}$ of $\Gamma$ of  congruence type  with  exponent  $\alpha\in (0,1)$,  denote by $\pi_{i}: M_{i}\rightarrow M$ the regular Riemannian cover of $M$ associated to $\Gamma_{i}$. For any non-negative integer $ k$ such that  
\[
\epsilon_{n, k}:=(n-1)-2k(1+\epsilon)> 0,
\] 
and for all $i$ such that
\[
\gamma_{g_{i}}(M_{i})>1+\frac{\ln(2)}{\epsilon_{n,k}},
\]
we have
\[
b_{k}(M_{i})\leq k(n, k, \epsilon, \Gamma) Vol(M_{i})^{1- \epsilon_{n, k}\alpha},
\]
where $k(n, k, , \epsilon, \Gamma)$ is a positive constant. Moreover, if $\epsilon>0$, then for $k$-forms of the critical degree such that $\epsilon_{n, k}=0$, we have
\[
b_{k}(M_{i})\leq l(n, k, \epsilon, \Gamma) \frac{Vol_{g_{i}}(M_{i})}{\alpha\ln(Vol_{g_{i}}(M_{i}))},
\]
for some positive constant $l(n, k, \epsilon, \Gamma)$.
\end{corollary}

\begin{proof}
Re-express the functions of injectivity radius in Theorems \ref{firstattempt} and \ref{Price3} as functions of $Vol(M)$.  \end{proof}

\begin{remark}
We note that Corollaries \ref{Pinched-vanishing} and \ref{pinchedalpha} can alternatively be derived by combining Theorem 3.2 in \cite{Donnelly} with Theorem 0.1 in \cite{Clair}. Finally, we point out that Corollary \ref{Pinched-vanishing} can also be derived by combining L\"uck approximation result with the Donnelly-Xavier vanishing.
\end{remark}

When $M=\IH^{n}_{\IR}/\Gamma$ is a compact hyperbolic manifold of dimension $n=2k+1$, we can also extend the results of Corollary \ref{2attempt} to estimate $b_k(M)$. In this case, in order to apply the usual Price inequality approach we need to understand the magnitude of $\mu$.

\begin{corollary}\label{Mark2} 
Let $M$ be a compact hyperbolic manifold of dimension $2k+1$ with $\gamma_g(M) >1$. There exists $C_k>0$ such that 
\begin{align}b_k(M)\leq C_k\frac{Vol(M)}{\gamma_g(M) -1}.
\end{align}
\end{corollary}
\begin{proof}
 The proof is the same as Corollary \ref{2attempt}, except that we now have 
\[
q(r) = \mu(r)\coth(r), 
\]
so that we need to estimate the magnitude of  $\mu$ from below. Let $h$ be a harmonic $k$ form, $k\not = 0, n$, not vanishing at $p\in M$. Recall that we have $\mu(0) = \frac{k}{2k+1}$, see Lemma \ref{muzero}. In hyperbolic geometry, it is possible to show there is a $k$ dependent positive lower bound $\tilde\mu$ for $\mu$ on $[0,1]$. This lower bound can be explicitly estimated using separation of variables. Relying less on the special geometry, we may also estimate $\mu$ as follows. Let $L\in (0,1]$. Identify $B_R\ssm\{0\}$ with $(0,R)\times S^{n-1}$ in the usual manner, and write $h = h_0+dr\wedge h_1,$ with $i_{\p_r}h_j = 0$. Use the product structure to identify the $h_j$ with a one parameter family of forms on $S^{n-1}$.  Set 
\begin{align}b_L(r) = \int_L^rh_1(s)ds.
\end{align}  
Then since $dh=0$ we have 
$$db_L(r) = h(r)-h_0(L).$$ 
The closed form $h_0(L)$ is exact on $S^{n-1}$. Hence there exists a $(k-1)$-form $\beta_0$ on $S^{n-1}$ such that 
$$d\beta_0 = h_0(L).$$
Moreover, 
$$\int_{S_L}|\beta_0|^2d\sigma \leq c_{n,k}^2\sinh^2(L)\int_{S_L}|h_0|^2d\sigma,$$
for some $c_{n,k}>0$. 
On the other hand, we have 
\begin{align}\label{smallmu}
\int_{B_L(p)}|h|^2dv & = \int_{B_L(p)}\langle d(b_L+\beta_0),h\rangle dv \nonumber\\
&= \int_{S_L(p)}\langle \beta_0,e^*(dr)h\rangle d\sigma \leq \sqrt{\mu(L)}c_{n,k}\sinh(L)\|h\|_{S_L}^2\nonumber\\
&=\frac{2\sqrt{\mu(L)}}{1-2\mu(L)}c_{n,k}\sinh(L)\int_{B_{L}(p)} q(r)|h|^2 dv,
\end{align}
where we have used \eqref{lie4} for the last equality. By Lemma \ref{Moser}
\begin{align}\label{L1}
\sinh(L)\int_{B_{L/2}(p)} q(r)|h|^2 dv &\leq d(n)\Big(1+\frac{2}{L}\Big)^{n}\sinh(L)\int_{B_{L}(p)} |h|^2 dv\int_{B_{L/2}(p)} q(r)dv\nonumber\\
&\leq d(n, L)\int_{B_{L}(p)} |h|^2 dv.
\end{align}
where
\[
d(n, L):=d(n)\Big(1+\frac{2}{L}\Big)^{n}\sinh(L)Vol(S^{n-1})\sinh(L/2)^{n-1}.
\]
On the other hand
\begin{align}\label{L2}
\sinh(L)\int_{B_{L}\ssm B_{L/2}}q(r)|h|^{2}dv\leq (n-1)\sinh(L)\coth(L/2)\int_{B_{L}}|h|^{2}dv.
\end{align}
By combining Equations \eqref{L1} and \eqref{L2}
\begin{align}\notag
\sinh(L)\int_{B_{L}(p)} q(r)|h|^2 dv  \leq C_{1}(n,L)\int_{B_{L}(p)} |h|^2 dv, 
\end{align}
where $C_1(n,L)$ is a constant depending only on $n$ and $L$.
Inserting this estimate back into \eqref{smallmu} gives 
\begin{align}\label{smallmu2}
1 \leq \frac{2\sqrt{\mu(L)}}{1-2\mu(L)}c_{n,k} C_1(n,L) .
\end{align}
Hence $\mu(L)$ is bounded below for $L$ bounded. Given this lower bound, the proof proceeds exactly as in Corollary \ref{2attempt} and Theorem \ref{Price3}. 
\end{proof} 

\begin{remark}
One of the referees brought to our attention the fact that a slightly weaker version of Corollary \ref{Mark2} could be derived from the computation of the Novikov-Shubin invariant for real hyperbolic manifolds given in \cite{Olbrich}. 
\end{remark}

We summarize our  Betti number estimates for real hyperbolic manifolds with the following corollary.

\begin{corollary}\label{real hyperbolic}
Let $X^{n}=\IH^{n}_{\IR}/\Gamma$ be a closed real hyperbolic manifold with $sec_{g}=-1$ and injectivity radius $\gamma_{g}(X)>1$. Given a cofinal filtration $\{\Gamma_{i}\}$ of $\Gamma$, let us denote by $\pi_{i}: X_{i}\rightarrow X$ the regular Riemannian cover of $X$ associated to $\Gamma_{i}$. If $n=2m$, for any integer $1\leq k<m$,  there exists a positive constant $c_{1}(n, k)$ such that
\[
\frac{b_{k}(X_{i})}{Vol_{g_{i}}(X_{i})}\leq c_{1}(n, k)e^{-(1+2(m-k))\gamma_{g_{i}}(X_{i})}.
\]
On the other hand, if $n=2m+1$, for any integer $k<m$ there exists a positive constant $c_{2}(n, k)$ such that
\[
\frac{b_{k}(X_{i})}{Vol_{g_{i}}(X_{i})}\leq c_{2}(n, k)e^{-(2(m-k))\gamma_{g_{i}}(X_{i})}.
\]
In both cases, the sub volume growth of the Betti numbers along the tower of coverings is exponential in the injectivity radius. Finally, for $n=2k+1$ we have the existence of a positive constant $c_{3}(n, k)$ such that
\[
\frac{b_{k}(M_{i})}{Vol_{g_{i}}(M_{i})}\leq \frac{c_{3}(n, k)}{\gamma_{g_{i}}(M_{i})-1}.\\
\]
\end{corollary}

\begin{proof}
For $n=2m$, notice that $\epsilon_{n, k}=\epsilon_{2m, m-1}=1$. For $n=2m+1$, we have $\epsilon_{n, k}=\epsilon_{2m+1, m}=0$. The statement in this case then follows from Corollary \ref{firstattempt}. For the critical case of $k$-forms in dimension $n=2k+1$, we use Corollary \ref{Mark2}.
\end{proof}

This result can also be obtained by trace formula techniques, but we consider our proof to be significantly simpler. For more details see again \cite{Xue}, \cite{Sarnak} and \cite{Marshall}.
 
\section{$L^2$-Cohomology and $L^{2}$-Betti Numbers}\label{Atiyah}

Let $H^{k}_{2,N}(\Omega)$ denote the (absolute) $L^2-$ cohomology of $\Omega$. Let $\Delta_k$ denote the Laplace Beltrami operator on $k-$forms. If $0$ is not in the essential spectrum of $\Delta_k$, 
then
\begin{align}\label{gap}\mathcal{H}^{k}_{2,N}(\Omega)\simeq H^{k}_{2}(\Omega).
\end{align} 

\begin{lemma}\label{B}
Let $(M^{n}, g)$ be a simply connected non-compact complete Riemannian manifold without conjugate points and $-1\leq \sec_{g}\leq 1$. If there exists $\delta>4k^2$ such that
\[
-Ric \geq \delta g,
\]
then zero is not in the essential spectrum of $\Delta_{k}$.
\end{lemma}

\begin{proof}
It is well-known that zero is not in the essential spectrum of $\Delta_{k}$ if and only if there is a compact set $K\subset M$ and a constant $\gamma>0$ such that
\[
||d\alpha||^{2}_{L^{2}}+||d^{*}\alpha||^{2}_{L^{2}}\geq\gamma||\alpha||^{2}_{L^{2}},
\]
for any smooth $k$-form $\alpha$ compactly supported in $M\ssm K$. (See for example \cite{Anghel}.) Fix $p\in M$. Choose $\rho$ large enough so that $\Big(\frac{\sqrt{\delta}}{2}-k\coth(\rho)\Big)=:\epsilon$ is strictly positive. 
Given $\alpha\in C^{\infty}_{c}(\Lambda^{k}T^{*}(M\ssm \overline{B_{\rho}(p)}))$, choose $R$ sufficiently large so that the support of $\alpha$ is contained in $B_{R}(p)$. Now, in the absence of the harmonicity assumption and with the addition of the support assumption,  \eqref{lie4} becomes
\[
\int_{B_{R}}(i_{-\partial_{r}}\alpha, d^{*}\alpha)dv+\int_{B_{R}}(i_{-\partial_{r}}d\alpha, \alpha)dv\geq\int_{B_{R}}q(r)|\alpha|^{2}dv.
\]
Since $|\partial_{r}|=1$,
\[
\int_{B_{R}}(i_{-\partial_{r}}\alpha, d^{*}\alpha)dv+\int_{B_{R}}(i_{-\partial_{r}}d\alpha, \alpha)dv\leq  \sqrt{2}\Big(||d\alpha||^{2}_{L^{2}}+||d^{*}\alpha||^{2}_{L^{2}}\Big)^{1/2}||\alpha||_{L^{2}}.
\]
On the other hand, since the support of $\alpha$ does not intersect the closure of the ball $B_{r_{\epsilon}}(p)$ we have that 
\[
\int_{B_{R}(p)}q(r)|\alpha|^{2}dv\geq \epsilon \int_{B_{R}(p)}|\alpha|^{2}dv.
\]
Setting $\gamma =\epsilon^2/2$ gives the desired lower bound on the spectrum. 
\end{proof}

\begin{corollary}\label{C}
Let $(M^{n}, g)$ be a simply connected non-compact complete Riemannian manifold without conjugate points and $-1\leq \sec_{g}\leq 1$. If there exists $\delta>4k^2$ such that
\[
-Ric\geq \delta g,
\]
then  $\mathcal{H}^{k}_{2}(M)=H^{k}_{2}(M)=0$. 

\end{corollary}
\begin{proof}
By Lemma \ref{B}, zero is not in the essential spectrum of $\Delta_{k}$. Thus, $\mathcal{H}^{k}_{2}(M)=H^{k}_{2}(M)$. The vanishing follows by applying Lemma \ref{A} and standard long exact sequences. 
\end{proof}

We now collect some consequences of Corollary \ref{C} regarding the vanishing of $L^{2}$-Betti numbers of certain classes of manifolds without conjugate points and with negative Ricci curvature. The $L^{2}$-Betti numbers are non-negative \emph{real} valued numerical invariants associated to a closed Riemannian manifolds. They were originally introduced by Atiyah in \cite{Atiyah} in connection with $L^{2}$-index theorems. Let us briefly recall their definition.

\begin{define} \cite[p. 44]{Atiyah} 
Let $(M^{n}, g)$ be a closed aspherical manifold. Let $\pi: (\tilde{M}, \pi^{*}(g))\rightarrow (M, g)$ be the Riemannian universal cover. Thus, $M=\tilde{M}/\Gamma$ where $\Gamma$ is a torsion free infinite group of isometries.  The $L^{2}$-Betti numbers of $M$ are the von Neumann dimension of the $\Gamma$-module $\mathcal{H}^{k}_{2}(\tilde{M}):$
\[
b^{(2)}_{k}(M):=\dim_{\Gamma}(\mathcal{H}^{k}_{2}(\tilde{M})).
\] 
\end{define}

We do not enter here into a detailed discussion of the theory of von Neumann dimension of Hilbert spaces with group actions. For our purposes, it suffices to recall (see Atiyah \cite{Atiyah}, see also \cite[Chapter I]{LuckBook}) that
\begin{align}\label{G-dimension}
\dim_{\Gamma}\mathcal{H}^{k}_{2}(\tilde{M})=0 \quad \Leftrightarrow \quad \mathcal{H}^{k}_{2}(\tilde{M})=0.
\end{align}
 
We apply our Price inequality to prove a vanishing result for $L^{2}$-Betti numbers of manifolds without conjugate points and negative Ricci curvature.

\begin{theorem}\label{Betti-Vanishing}
Let $(M^{n}, g)$ be a closed Riemannian manifold without conjugate points and $-1\leq \sec_{g}\leq 1$. If there exists $\delta>4k^2$ such that
\[
-Ric\geq \delta g,
\]
then $b^{(2)}_{k}(M)=0$. 
\end{theorem}
\begin{proof}
The claimed vanishing follows from Corollary \ref{C} combined with \eqref{G-dimension}. 
\end{proof}

Theorem \ref{Betti-Vanishing} provides new evidence for the Singer Conjecture.  Let us recall its statement.

\begin{conjecture}[Singer Conjecture]\label{singer}
If $M^{n}$ is a closed aspherical manifold, then
\[
b^{(2)}_{k}(M)=0, \quad \textrm{if}\quad 2k\neq n.
\]
\end{conjecture}

This conjecture is still open, even under the assumption that $M^{n}$ admits a metric with \emph{strictly negative} sectional curvature. While Theorem \ref{Betti-Vanishing} does not assume the sectional curvature to be negative, it covers an insufficient range of Betti numbers to settle the conjecture. For more on the Singer Conjecture we refer to \cite[Chapter XI]{LuckBook}.

\begin{remark}
It is interesting to observe that Theorem \ref{Betti-Vanishing} when combined with the L\"uck approximation theorem \cite{Luck} can be used to give an alternative proof of Equation \eqref{limble291} in Theorem \ref{DGW}. More precisely, we need to apply first Theorem \ref{Betti-Vanishing} to closed manifolds with residually finite fundamental group, and then appeal to L\"uck's approximation theorem. Nevertheless, this $L^{2}$-cohomology approach to Theorem \ref{DGW} has the disadvantage of not estimating how fast the ratio in \eqref{limble291} converges to zero. On the other hand, our original Price inequality approach to Theorem \ref{DGW} provides directly an effective estimate for this convergence.\\
\end{remark}

We conclude with an $L^{2}$-Betti number vanishing result for $\epsilon$-pinched negatively curved closed manifolds. This result extends to the  case $\epsilon_{n,k}=0$ a vanishing theorem for $\epsilon$-pinched negatively curved manifolds given by Donnelly-Xavier in \cite[Proposition 4.1]{Donnelly}. \\
\begin{corollary}
Let $(M^{n}, g)$ be a complete simply connected Riemannian manifold of dimension $n_{\IR}\geq 4$. Assume the sectional curvature is $\epsilon$-pinched
\[
-(1+\epsilon)^{2}\leq \sec_{g}\leq -1,
\]
with $\epsilon>0$. Let $k$ be a positive integer such that
\[
\epsilon_{n, k}=(n-1)-2k(1+\epsilon)=0. 
\]
Then, there are no $L^{2}$-harmonic $k$-forms $\mathcal{H}^{k}_{2}(M)=0$.
\end{corollary}
\begin{proof}
This vanishing result is an immediate consequence of the Price inequality given in Theorem \ref{Price3} combined with the fact that $\gamma_{g}(M)=\infty$.
\end{proof}

As before, this result implies a vanishing result for $L^{2}$-Betti numbers of certain negatively curved manifolds.\\

\begin{proposition}\label{critical}
Let $(M^{n}, g)$ be a closed Riemannian manifold of dimension $n\geq 4$. Assume the sectional curvature is $\epsilon$-pinched
\[
-(1+\epsilon)^{2}\leq \sec_{g}\leq -1,
\]
with $\epsilon>0$. Let $k$ be a positive integer such that
\[
\epsilon_{n, k}=(n-1)-2k(1+\epsilon)=0,
\]
we have $b^{(2)}_{k}(M)=0$.\\
\end{proposition}

This vanishing result complements the one proved by Donnelly-Xavier in \cite{Donnelly} by extending it to the critical equality case. More precisely, they prove a vanishing for $L^{2}$-Betti numbers of any degree $k$ such that $\epsilon_{n, k}>0$. The vanishing in the hyperbolic case ($\epsilon=0$) for the critical degree $k=\frac{n-1}{2}$ was treated earlier by Dodziuk in \cite{Dodziuk}.  Alternatively, if one wishes, Corollary \ref{Mark2} can be used to give an alternative proof of Dodziuk's vanishing in the critical degree. Once again, the vanishing in Proposition \ref{critical} supports Conjecture \ref{singer}. \\ \\


\end{document}